\documentclass{article}
\usepackage{amssymb,amsmath,latexsym,amsthm,rotating}
\usepackage{vaucanson-g}

\usepackage{amssymb,epsfig}
\usepackage[cp 1250]{inputenc}
\usepackage{a4}
\usepackage{graphics}

\newtheorem{theorem}{Theorem}[section]
\newtheorem{lemma}[theorem]{Lemma}
\newtheorem{proposition}[theorem]{Proposition}

\newtheorem{corollary}[theorem]{Corollary}

\theoremstyle{definition}
\newtheorem{remark}[theorem]{Remark}
\newtheorem{example}[theorem]{Example}
\newtheorem{definition}[theorem]{Definition}

\numberwithin{equation}{section}

\def\Z{\mathbb Z}
\def\N{\mathbb N}

\def\C{\mathbb C}
\def\Q{\mathbb Q}

\def\Select{\mathrm{Select}}

\newcommand{\decdot}{\raisebox{0.1ex}{\textbf{.}}}

\begin{document}

\title{Two applications of the spectrum of numbers}

\author{Christiane Frougny\footnote{ 
IRIF, UMR~8243 CNRS and Universit\'e Paris-Diderot} 
and Edita Pelantov\'a\footnote{
Doppler Institute for Mathematical Physics and
Applied Mathematics, and
Department of Mathematics, FNSPE,
Czech Technical University in Prague}}
\date{ }

\maketitle

\begin{abstract}
Let the base $\beta$ be a complex number, $|\beta|>1$, and let $A \subset \C$ be a finite alphabet of digits.
The \emph{$A$-spectrum} of $\beta$ is the set
$S_{A}(\beta) = \{\sum_{k=0}^n a_k\beta^k \mid n \in \mathbb{N}, \ a_k \in {A}\}$.
We show that the spectrum $S_{{A}}(\beta)$ has an accumulation point if and only if
$0$ has a particular $(\beta, A)$-representation, said to be \emph{rigid}.

The first application is restricted to the case that $\beta >1 $ and the alphabet is
$A=\{-M, \ldots, M\}$, $M \ge 1$ integer. We show that the set
$Z_{\beta,M}$
of infinite $(\beta, A)$-representations of $0$
is recognizable by a finite B\"uchi automaton if and only if the spectrum $S_A(\beta)$ has no accumulation point.
Using a result of Akiyama-Komornik and Feng, this implies that
$Z_{\beta,  M}$ is recognizable by a finite B\"uchi automaton for any
positive integer $M \ge \lceil \beta \rceil -1$ if and only if
$\beta$ is a Pisot number.
This improves the previous bound $M \ge \lceil \beta \rceil $.

For the second application the base and the digits are complex. We consider the on-line algorithm for division of Trivedi and Ercegovac generalized
to a complex
numeration system. 
In on-line arithmetic the operands and results are processed in a digit serial manner, starting with the most significant digit.
The divisor must be far from $0$, which means that no prefix of the $(\beta,A)$-representation of the divisor can be small.
The 
numeration system $(\beta,A)$ is said to \emph{allow preprocessing} if there exists a finite list of transformations on the divisor
which achieve this task.
We show that $(\beta,A )$  allows preprocessing if and only if
the spectrum $S_{{A}}(\beta)$ has no  accumulation point.

\bigskip


\end{abstract}
\section{Introduction}
 
The so-called {\em
beta-numeration} has been introduced by R\'enyi in \cite{Ren} and studied by Parry in \cite{Parry}
in the case that $\beta$ is a 
real number, $\beta>1$, and since then there are been many works
in this domain, in connection with number theory, dynamical systems, and automata theory, see the survey \cite{FrSak2010}  or more recent \cite{rigo}
for instance.

For $\beta>1$ and $M \ge 1$ an integer, the following spectrum
$$X_M(\beta)=\{ \sum_{k=0}^n a_k \beta^k \mid n \in \N, \; a_k \in \{0, 1,\ldots, M\}\}$$
has been introduced by Erd\H{o}s, Jo\'o and Komornik~\cite{EJK}.

Since $X_M(\beta)$ is discrete its elements can be arranged into an increasing sequence
$$0=x_0<x_1<\cdots$$
Denote $\ell_M(\beta)= \liminf_{k \to \infty} (x_{k+1}-x_{k})$. 
Numerous works have been devoted to the study of this value, see in particular the introduction and the results of \cite{AkiKomornik}.

More generally, let $\beta$ be a complex number, $|\beta|>1$, and let $A \subset \C$ be a finite alphabet of digits.
The \emph{$A$-spectrum} of $\beta$ is the set
$$
S_{A}(\beta) = \Bigl\{\sum_{k=0}^n a_k\beta^k \mid n \in \mathbb{N}, \ a_k \in {A}\Bigr\} .
$$
Recently Feng  answered an open question raised in~\cite{EJK}, see also \cite{AkiKomornik}, on the density of the spectrum of  $\beta$ when $\beta$ is real and the digits are consecutive integers:
 
\begin{theorem} [\cite{Feng}]\label{feng}   Let $\beta >1$ and let $A=\{-M, \ldots,M\}$, $M$ an integer $\ge 1$. 
Then  the spectrum $  S_A({\beta})$ is dense in $\mathbb{R}$ if and only if $\beta < M+1$ and $\beta$ is not a Pisot number. 
\end{theorem}
Feng has obtained the following corollary:
$\ell_M(\beta)=0$ if and only if $\beta<M+1$ and $\beta$ is not a Pisot number.

\bigskip

In this paper we use the concept of spectrum of a number to solve two problems arising in beta-numeration.

Let $\beta$ and the digits of $A$ be complex. The topological properties of the spectrum are linked with a particular representation of $0$.
Let $z_1z_2 \cdots$ be a $(\beta, A)$-representation of $0$, that is to say,
$\sum_{i\ge 1}z_i \beta^{-i}=0$. 
It is said to be \emph{rigid} if
$0\decdot z_1z_2 \cdots z_j \neq 0\decdot 0z'_2 \cdots z'_j$ for all $j \ge 2$ and for all $z'_2 \cdots z'_j$
in $A^*$. 
The term ``rigid" comes from the preprocessing motivation, see Section~5.

We first prove that the spectrum $S_{{A}}(\beta)$ has an accumulation point if and only if
$0$ has a rigid $(\beta, A)$-representation, Theorem~\ref{rigid}.

Then we obtain some results when the base is a complex Pisot number, which extend the real case covered by Garsia~\cite{Garsia}.
Let $\beta$ be a complex number, and let $A \subset \Q(\beta)$ containing $0$. 
If $\beta$ is real and if $\beta$ or  $-\beta$ is a Pisot number, 
or if $\beta \in \mathbb{C}\setminus \mathbb{R}$ is a complex Pisot number
then $S_{{A}}(\beta)$ has no accumulation point, Theorem~\ref{pisot}.

\bigskip

The first question we address in this work is the one 
of the recognizability by a finite B\"uchi automaton of the set of infinite $\beta$-representations of $0$
when $\beta$ is a real number and the digits are integer.

The set of infinite $\beta$-representations of $0$ on the alphabet $\{-M, \ldots, M\}$,
$M \ge 1$ integer, is denoted
$$Z_{\beta,M}=\{z_{1} z_{2}\cdots \mid \sum_{i\ge 1}z_i \beta^{-i}=0, \; z_i \in \{-M, \ldots, M\}\}.$$

The following result has been formulated in~\cite{FrSak2010}:
\begin{theorem}\label{origin}
Let $\beta >1 $. The following conditions are equivalent:
\begin{enumerate}
  \item the set $Z_{\beta,  M}$ is recognizable by a finite B\"uchi automaton for every integer 
$M$,
  \item the set $Z_{\beta,  M}$ is recognizable by a finite B\"uchi automaton for one integer 
$M \ge \lceil \beta \rceil$,
  \item $\beta$ is a Pisot number.
\end{enumerate}
\end{theorem}
(3) implies (1) is proved in \cite{Fr1992}, (1) implies (3) is proved in \cite{BeFr1994}
and (2) implies (1) is proved in \cite{FrSak1999}.

Note that in \cite{Bugeaud1996} Bugeaud has shown, using (1) implies (3) of Theorem~\ref{origin}, that if $\beta$ is not a Pisot number then there exists an integer $M$ such that
$\ell_M(\beta)=0$.

In this paper we first prove that 
the set $Z_{\beta,M}$ is recognizable by a finite B\"uchi automaton  if and only if the spectrum $S_{  A }(\beta)$  has no accumulation point, Theorem~\ref{recog2}. 

By \cite{AkiKomornik} or \cite{Feng} it is known that, for $A  =\{ -M, \ldots,M\}$, the spectrum   $S_{  A }(\beta)$ has an accumulation point if and only if   $\beta  < M+1$ and $\beta$ is not Pisot. 

This result together with Theorem~\ref{recog2} proves the conjecture stated  in~\cite{FrSak2010}: \\
If the set  $Z_{\beta, \lceil \beta \rceil -1}$ is recognizable by a finite B\"uchi automaton
then $\beta$ must be a Pisot number.

Moreover we obtain a simpler proof of the implication $(2)\Rightarrow (3)$ of   Theorem~\ref{origin}.
Note that the value  $M= \lceil \beta \rceil-1$  is the best possible as $Z_{\beta, M}$  is reduced to $\{0^\omega\}$ if $M< \lceil \beta \rceil-1$.

\medskip

Normalization in base $\beta$ is the function which maps any $\beta$-representation on the canonical alphabet
$A_\beta=\{0,\ldots,\lceil \beta\rceil-1\}$ of a number $x \in [0,1]$ onto the greedy
$\beta$-expansion of $x$. Since the set of greedy $\beta$-expansions of the elements of $[0,1]$
is computable by a finite B\"uchi automaton when $\beta$ is a Pisot number, see~\cite{BertrandMathis86},
the following result holds true:\\
Normalization in base $\beta>1$ is computable by a finite  B\"uchi automaton
on the alphabet $A_\beta \times A_\beta$ if and only if $\beta$ is a Pisot number.

\bigskip

The second utilisation of the notion of spectrum occurs in the on-line algorithm for 
division in a complex base.

On-line arithmetic, introduced in~\cite{TrivediErcegovac} for an integer base, is a mode of computation where operands and results are processed in a digit serial manner, starting with the most significant digit. To generate the first digit of the result, the first $\delta$ digits of the operands are required. The integer $\delta$ is called the delay of the algorithm.
One of the interests of the functions that are on-line computable is that they are continuous for the usual topology on the set of infinite words on a finite alphabet.

In \cite{BrFrPeSv1, BrFrPeSv2} we have extended the original on-line algorithm of Trivedi-Ercegovac to a complex base.
The algorithm for on-line division in a complex numeration system $(\beta,A)$ has two parameters: the delay $\delta\in \N$ 
and  $D  >0$,  the minimal value (in modulus) of the divisor.

When making division, we need that the divisor stays away from $0$.
By definition of the on-line algorithm, this means that the value of all the prefixes of the divisor
$d_1d_2 \cdots$ must be greater in absolute value than $D>0$, so the divisor must be preprocessed before making the division.

We say that  a complex numeration system $(\beta,A )$  \emph{allows preprocessing}
if there exists a finite list of transformations on the $(\beta,A )$-representation of the divisor
which achieve this task, see Definition~\ref{DefPrepro}.

We show that a complex numeration system $(\beta,A )$  allows preprocessing if and only if
the spectrum $S_{{A}}(\beta)$ has no  accumulation point, Theorem~\ref{prepro}.

\section{Preliminaries}

\subsection{Words and automata}
Let $A$ be a finite alphabet. A \emph{finite word} $w$ on $A$ is a finite concatenation of letters from
$A$, $w=w_1 \cdots w_n$ with $w_i$ in $A$.
The set of all finite
words over $A$ is denoted by $A^*$.
An  \emph{infinite word} $w$ on $A$ is an infinite concatenation of letters from
$A$, $w=w_1 w_2\cdots $ with $w_i$ in $A$.
The set of all infinite
words over $A$ is denoted by $A^\N$. The infinite concatenation $uuu\cdots$ is noted $u^\omega$.
If $w=uv$, $u$ is a \emph{prefix} of $w$.

An \emph{automaton} $\mathcal{ A}= (A, Q, I, T)$ over the alphabet $A$ is a directed graph labeled by letters of
$A$, with a denumerable set $Q$ of vertices called \emph{states}. $I \subseteq Q$ is the set of \emph{initial}
states, and $T \subseteq Q$ is the set of \emph{terminal} states. The automaton is said to be \emph{finite}
if the set of states $Q$ is finite.

An infinite path of $\mathcal{ A}$ is said to be \emph{successful} if it starts in $I$ and goes infinitely often through $T$.
The set of infinite words \emph{recognized} by $\mathcal{ A}$ is the set of labels of its successful infinite paths.
An automaton used to recognize infinite words in this sense is called a \emph{B\"uchi automaton}.

\subsection{Numeration}\label{betanum}

Let $\beta$ be a complex number, $|\beta|>1$, and let $A \subset \C$ be a finite set, the alphabet of digits.
We say that $(\beta, A)$ is a \emph{numeration system}.
A \emph{$(\beta,A)$-representation} of a number $z$ is an infinite word $z_1z_2 \cdots$ such that $z=\sum_{i=1}^{+\infty} z_i \beta^{-i}$ with $z_i$ in $A$. It should be noted that here we do not make any hypothesis on the fact that every complex number has, 
or does not have, a $(\beta,A)$-representation. This is a difficult problem, studied by many authors, see the pioneering works of Knuth~\cite{Knuth}, K\'atai and Kov\'acs~\cite{KataiKovacs}, Gilbert~\cite{Gilbert}, Thurston~\cite{Thurston} for instance.

\medskip

We now recall some definitions and results on the so-called {\em
beta-numeration}, see~\cite{FrSak2010} or~\cite{rigo} for a survey.  Let $\beta>1$ be a real number.  Any real number
$x\in[0,1]$ can be represented by a greedy algorithm as
$x=\sum_{i=1}^{+\infty} x_i \beta^{-i}$ with $x_i$ in the \emph{canonical alphabet} $A_\beta
=\{0,\ldots,\lceil \beta\rceil-1\}$ for all $i\ge 1$. The greedy
sequence $(x_i)_{i\ge 1}$ corresponding to a given real number $x$ is
the greatest in the lexicographical order, and is said to be the {\em
  $\beta$-expansion} of $x$, see \cite{Ren}. It is denoted by
$d_\beta(x)= (x_i)_{i\ge 1}$.  When the expansion ends in infinitely
many $0$'s, it is said to be {\em finite}, and the $0$'s are omitted.

The greedy $\beta$-expansion of $1$ is denoted $d_\beta(1)=(t_i)_{i \ge 1}$. 
When it is finite, of the form 
$d_\beta(1)=t_1 \cdots t_m$, the {\em quasi-greedy} $\beta$-expansion of $1$ is defined
as $d_\beta^*(1)=(t_1 \cdots t_{m-1}(t_{m}-1))^\omega$.
If it is infinite, set $d_\beta^*(1)=d_\beta(1)$.
The sequence
$d^*_{\beta}(1)$ is the lexicographically greatest infinite representation of 1 in the base $\beta$ and the alphabet $\mathbb{N}$.  
It is known from~\cite{Parry} that a sequence of integers $x_1x_2 \cdots$ is the greedy $\beta$-expansion of some $x$
from $[0,1]$ if and only if, for all $j \ge 1$, $x_jx_{j+1} \cdots$ is less than or equal
to $d^*_{\beta}(1)$ in the lexicographic order.

\medskip

{\noindent}Notation: The numerical value
$y_{m-1} \beta^{m-1} +\cdots +y_0 + y_{-1}\beta^{-1} + y_{-2}\beta^{-2} +\cdots$
is denoted by $y_{m-1} \cdots y_0 \decdot y_{-1}y_{-2}\cdots$.

\subsection{Numbers}

A number $\beta>1$ such that $d_\beta(1)$ is eventually periodic is a \emph{Parry number}.
It is a \emph{simple} Parry number if $d_\beta(1)$ is finite.

A \emph{Pisot number} is an algebraic integer greater than $1$ such that
all its Galois conjugates have modulus less
than $1$. Every Pisot number is a Parry number, see~\cite{Bertrand77} and~\cite{Schmidt}.

A \emph{complex Pisot number} is an algebraic integer $\beta$ such that $|\beta|>1$ and such that
all its Galois conjugates different from its complex conjugate $\overline{\beta}$ have modulus less
than $1$.
\section{Spectrum and rigid representation of $0$}

Let $\beta$ be a complex number, $|\beta|>1$, and let $A \subset \C$ be a finite alphabet.
We introduce the ${A}$-spectrum  of $\beta$ as
$$
S_{{A}}(\beta) = \Bigl\{\sum_{k=0}^n a_k\beta^k \mid n \in \mathbb{N}, \ a_k \in {A}\Bigr\} .
$$

The topological properties of $S_{A}(\beta)$ are linked with a particular representation of $0$.

\begin{definition}\label{RigidZero}
Let $z_1z_2 \cdots$ be a $\beta$-representation of $0$ on $A$, that is to say,
$\sum_{i\ge 1}z_i \beta^{-i}=0$. 
It is said to be \emph{rigid} if
$0\decdot z_1z_2 \cdots z_j \neq 0\decdot 0z'_2 \cdots z'_j$ for all $j \ge 2$ and for all $z'_2 \cdots z'_j$
in $A^*$.
\end{definition}

\begin{example}The signed digit $(-1)$ is denoted $\overline{1}$.
In base $2$ with alphabet $\{\overline{1}, 0,1\}$, $0$ has
two representations, namely $0=0\decdot 1\overline{1}\,\overline{1}\,\overline{1}\,\overline{1}\cdots = 0\decdot  \overline{1}\,1\,1\,1\,1\cdots $.
They are not rigid, since $0 \decdot 1\overline{1} = 0\decdot 01$ and  $0 \decdot \overline{1}1 = 0\decdot 0\overline{1}$.
\end{example}

\begin{definition}\label{tail}
Let $z_1z_2 \cdots$ be a $(\beta,A)$-representation of $0$.
For $n$ in $\N$, its \emph{$n$-th tail} is
$r_n = 0\decdot z_{n+1}z_{n+2}z_{n+3}\cdots$.
\end{definition}

\begin{lemma}\label{aux} Let $z_1z_2 \cdots$ be a $(\beta,A)$-representation of $0$. 
\begin{enumerate}
\item If the  sequence $(r_n)_{n \in \N}$ is injective, then the spectrum $S_{  A }(\beta)$   has an accumulation point. 
\item  If  the representation  of $0$ is rigid, then  the sequence $(r_n)_{n \in \N}$ is injective. 
\end{enumerate}
\end{lemma}

\begin{proof} 
Since $0=0\decdot z_1z_2z_3\cdots$,   the  $n^{th}$ tail $r_n= \sum\limits_{k=1}^{+\infty} z_{n+k}\beta^{-k} = -\sum\limits_{k=0}^{n-1} z_{n-k}  \beta^k  $. It means that $-r_n$ belongs to the spectrum  $S_{  A }(\beta)$ and moreover $$|r_n|\leq \frac{\alpha}{|\beta|-1}, \quad \text{ where }  \ \alpha=\max \{|a|\,:\, a\in A\}.$$ 

1) \ If the sequence $(r_n)_{n \in \N}$ is injective, then  the ball centered at  $0$  with radius  $\frac{\alpha}{|\beta|-1} $ contains infinitely many elements $(-r_n)$  of the spectrum, and thus the spectrum has an accumulation point.

2) \ Suppose that the representation  of $0$ is rigid.  We show by contradiction  the  injectivity of $(r_n)_{n \in \N}$. Let us assume that  $r_i=r_j$ for some indices $i<j$.  Then    
$\sum_{k=0}^{j-1} z_{j-k}  \beta^k = \sum_{k=0}^{i-1} z_{i-k}  \beta^k$ and thus 
$0\decdot z_1z_2 \cdots z_j =0\decdot \underbrace{0\cdots 0}_{(j-i)\,\text{times}} z_1 \cdots z_i$ ---  a contradiction with the rigidity of the representation of zero.
\end{proof}

\begin{theorem}\label{rigid}
Let $\beta$ be a complex number, $|\beta|>1$, and let $A \subset \C$ be a finite alphabet.
The spectrum $S_{{A}}(\beta)$ has an accumulation point if and only if
$0$ has a rigid $(\beta, A)$-representation.
\end{theorem}

\begin{proof}  $(\Rightarrow)$ \ \ 
 Let $s$ be an accumulation point of $S_{  A }(\beta)$.  There exists an injective sequence $(x^{(n)})_{n \in \N}$  of points from  $S_{  A }(\beta)$ such that $ \lim (x^{(n)})_{n \in \N}= s $. 
 For any $x\in  S_{  A }(\beta)$ denote $$\rho(x) = \min\{n \in \mathbb{N} : x = \sum_{k=0}^n a_k\beta^k , \text{\ with\ } a_k \in   A \}.$$  
Set $\rho_n =  \rho(x^{(n)}) $, then $x^{(n)}= \sum_{k=0}^{\rho_n} x_k^{(n)}\beta^k  $. The sequence  $(\rho_n)_{n \in \N}$  is unbounded, as  there exists only a finite number of strings of a given length over a finite alphabet.    Without loss of generality assume that $(\rho_n)_{n \in \N}$ is strictly increasing.    Clearly,
\begin{equation}\label{toZero}
\frac{x^{(n)}}{\beta^{1+\rho_n}} = 0\decdot x_{\rho_n }^{(n)} \cdots x_2^{(n)}x_1^{(n)} x_{0 }^{(n)}0000\cdots  \to 0
\end{equation}
 since the nominators tend to $s$.
The fact that $  A ^\mathbb{N}$ endowed with the product topology is a  compact  space implies the
existence of a string $x_1x_2x_3 \cdots $ which is the limit of a subsequence   of  $(x_{\rho_n }^{(n)} \cdots x_2^{(n)}x_1^{(n)} x_{0 }^{(n)}0^\omega)_{n \in \N}$. 
It means that for any $N\in \mathbb{N}$  one can find $n \in \mathbb{N}$ such that   $\rho_n >N$  and $x_{\rho_n }^{(n)} \cdots x_2^{(n)}x_1^{(n)} x_{0 }^{(n)}$ is a prefix of $x_1x_2x_3 \cdots $.  The definition of $\rho_n$  and the fact \eqref{toZero} forces  $0\decdot x_1x_2x_3 \cdots $ to be a rigid representation of 0.

\medskip

$(\Leftarrow)$ \ \   Let $0 = 0\decdot z_1z_2 z_3\cdots $ be a rigid representation of zero. 
Then by  Point 2 of Lemma \ref{aux}, the sequence of its  tails is injective and by Point 1 of the same lemma, the spectrum has an accumulation point.   
\end{proof}

We now turn to the Pisot case. The real case is due to Garsia~\cite{Garsia}, and we follow his idea.

\begin{theorem}\label{pisot}
Let $\beta$ be a complex number, $|\beta|>1$, and let $A \subset \Q(\beta)$ containing $0$. 
\begin{enumerate}
  \item If $\beta$ is real and if $\beta$ or  $-\beta$ is Pisot
  \item or if $\beta \in \mathbb{C}\setminus \mathbb{R}$ is complex Pisot
\end{enumerate}
then $S_{{A}}(\beta)$ has no accumulation point.
\end{theorem}

\begin{proof}
Let $\beta = \beta_1$  be a complex Pisot number of degree $r$ with conjugates $\beta_2=\overline{\beta_1}, \beta_3, \ldots, \beta_r$, i.e. $|\beta_k|<1$ for $k=3,4,\ldots, r$. We denote $\sigma_k: \mathbb{Q}(\beta_1) \to \mathbb{Q}(\beta_k)$ the isomorphism induced by $\beta_1 \mapsto \beta_k$.  As $  A $ is finite there exists $q \in \mathbb{N}$ such that  $q  A $ belongs to the ring of integers of the field $\mathbb{Q}(\beta)$.  In particular, the norm $N(qa) = q^r \prod_{k=1}^r|\sigma_k(a)|$ is 
an integer for any letter $a$ in $A$.  

Consider $ x, y \in  S_{  A }(\beta)$, $x\neq y$.  Then the difference between $x$ and $y$ can be expressed 
as $x-y=v = \sum_{j=0}^n b_j\beta^j$, for some $n$ in $\N$ and $b_j$ in   $A  -  A $.

Let us denote $A_k = \max\{|\sigma_k(a)|: a\in   A \}$.  
For $k=3,4,\ldots,r$, the modulus of the $k$-th conjugate of $v$  satisfies 
$$|\sigma_k(v)| \leq \sum_{j=0}^n |b_j|.|\beta_k|^j \leq 2A_k\sum_{j=0}^\infty|\beta_k|^j = 2A_k\tfrac{|\beta_k|}{1-|\beta_k|}.$$
 Since $\beta$ and $qb_k$ are algebraic integers, $qv$ is an
algebraic integer as well and its norm is  a rational non-zero integer.  Compute  the norm  of $qv$
$$
1\leq |N(qv)| = q^r\prod_{k=1}^r|\sigma_k(v)| \leq q^r  v \,\overline{v}\,\prod_{k=3}^r|\sigma_k(v)|\leq  (2q)^r v\overline{v}\prod_{k=3}^r \frac{A_k|\beta_k|}{1-|\beta_k|} \,.
$$
It means that the squared distance $v\overline{v}$ of two different points from the spectrum  $S_{  A }(\beta)$ is bounded from below  by the constant  $(2q)^{-r}\prod_{k=3}^r \frac{1-|\beta_k|}{A_k|\beta_k|}$\,. Consequently, the spectrum has no accumulation point.

The case $\beta$ real is analogous. 
\end{proof}

If  the base $\beta$  is real and  the alphabet is a  symmetric set of consecutive integers,   Theorem~\ref{rigid}  together with  the following  theorem  answers completely  the question of the existence of a rigid representation of zero.   

\begin{theorem}[Akiyama and  Komornik \cite{AkiKomornik}, Feng \cite{Feng}]\label{Aki} 
Let $\beta>1$ and let
$A  =\{ -M, \ldots,M\}$. Then   $S_{  A }(\beta)$ has an accumulation point if and only if   $\beta  < M+1$ and $\beta$ is not Pisot. 
\end{theorem}

If the base $\beta$ is real but the alphabet is not symmetric we have only the following  partial observation.

\begin{proposition}\label{non-sym}
Let $\beta >1$ and  $\{-1,0,1\} \subset   A  =  \{m,\ldots , 0,\ldots,  M\}\subset \mathbb{Z}$. 
\begin{enumerate}
\item Zero  has a non-trivial $(\beta,   A )$-representation if and only if
$\beta \leq \max \{M+1, -m+1\}\,.$
\item If  $\beta \leq\max \{M+1, -m+1\}\,, $ and $\beta$ is not a Parry number, then  zero  has a rigid $(\beta,   A )$-representation.
\end{enumerate}
\end{proposition}

\begin{proof}  Let 
Let  $d_\beta(1)=t_1t_2t_3\cdots $ be the greedy expansion of $1$. Then $\beta -1\leq t_1 < \beta$,  $t_i\leq t_1$  and 
  $$0=0\decdot\overline{1}t_1t_2t_3\cdots =0\decdot 1\overline{t_1}\,\overline{t_2}\,\overline{t_3}\cdots$$
 We have  two non-trivial representations of $0$
 over the alphabets  $\{- \lceil \beta\rceil+1, \ldots, \overline{1},0,1\}$ and 
$\{\overline{1}, 0,1, \ldots, \lceil \beta\rceil-1 \}$ respectively.

Therefore, if  $\{-1, 0,1, \ldots,  t_1 \}\subset   A $ or $\{  - t_1,\ldots, -1, 0,1,\}\subset   A  $, zero has  a non-trivial   $(\beta,   A )$-representation.  Let us note that $ t_1 \in   A $ means $M\geq  t_1\geq \beta- 1$.   Similarly  $- t_1 \in   A $  implies  $m\leq - t_1\leq -\beta+1$. 

 On the other hand, let $M < \beta-1$ and  $m  > -\beta +1$.   Then  
for  $z = \sum_{k\geq 1} z_i\beta^{-i}$ with $z_i \in   A $ and  $z_1\geq 1$, we have 
$z\geq \tfrac1{\beta}+\sum_{i\geq 2}\tfrac{m}{\beta} =  \tfrac{\beta-1+m}{\beta(\beta -1)} >0$.  Analogously, if $z_1\leq -1$, then $z<0$. Consequently, $0$ has only the trivial representation.  
\medskip

Now assume that $\beta$ is not a Parry number.  Then the sequence of the $n^{th}$ tails of the
$\beta$-expansion of $1$, $r_n=0\decdot t_{n+1}t_{n+2}\cdots$,  is injective. By Lemma~\ref{aux} and Theorem~\ref{rigid}, zero has a rigid $(\beta,A )$-representation. 
\end{proof}

\begin{remark} A numeration system with negative base $-\beta<-1$ and an alphabet $A_{-\beta} = \{0,\ldots, \lfloor \beta \rfloor\}$  was introduced by  Ito and Sadahiro in~\cite{ItSa}.  Liao and Steiner in~\cite{{LiaoSt}} defined  an Yrrap number
 as an analogy of a Parry number for numeration systems with negative base.  This definition implies that if $\beta$ is not Yrrap, then there exists  a rigid  $(-\beta, A)$-representation of $0$ 
over the alphabet $A=\{1,\ldots, \lfloor \beta \rfloor +1\} $. 
\end{remark}


\section{A problem in automata theory}

\subsection{Representations of $0$}

Let $\beta$ be a real number $>1$
We consider infinite 
$\beta$-representations of $0$ on an alphabet of the form $\{-M, \ldots, M\}$, $M\ge 1$ integer. 
Let
$$Z_{\beta,M}=\{z_{1} z_{2}\cdots \mid \sum_{i\ge 1}z_i \beta^{-i}=0, \; z_i \in \{-M, \ldots, M\}\}$$
be the set of infinite words having value $0$ in base $\beta$ on the alphabet $\{-M, \ldots, M\}$.  

Proposition~\ref{non-sym} says that  $0$  has a non-trivial representation only if  $M\geq \lceil\beta\rceil -1$. Therefore, we consider only $M$ satisfying this inequality.

Note that, if $Z_{\beta,  M}$ is recognizable by a finite B\"uchi automaton, then, for every
 $c<M$, $Z_{\beta,  c}=Z_{\beta,  M} \cap \{-c, \ldots, c\}^\N$ is recognizable by a finite B\"uchi automaton as well.
 
\medskip

We briefly recall the construction of the (not necessarily finite) B\"uchi automaton recognizing $Z_{\beta,  M}$, see \cite{Fr1992} and
\cite{FrSak2010}:
\begin{itemize}
  \item the set of states is $Q_M \subset \{ \sum_{k=0}^n a_k \beta^k \mid n \in \N, \; a_k \in \{-M, \ldots, M\}\} \cap [-\frac{M}{\beta-1},\frac{M}{\beta-1}]$
  \item for $s, t \in Q_M$,  $a \in \{-M, \ldots, M\}$ there is an edge 
$$s \stackrel{a}{\longrightarrow} t \iff
t=\beta s +a$$
 \item the initial state is $0$
 \item all states are terminal.
  \end{itemize}

\begin{example}
Take $\beta=\varphi=\frac{1+ \sqrt{5}}{2}$ the Golden Ratio. It is a Pisot number, with $d_\varphi (1)=11$.
A finite B\"uchi automaton recognizing $Z_{\varphi,1}$ is designed in Figure~\ref{zer-aut-phi}.
The initial state is $0$, and all the states are terminal. 

\begin{figure}[h]
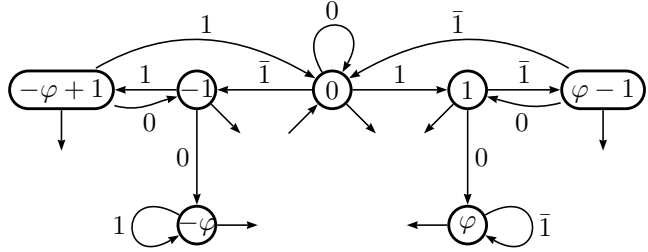

\begin{center}
\VCDraw{%
\begin{VCPicture}{(-5,-2.8)(4,2)}
    \MediumState
\State[0]{(0,0)}{A}
\Initial[sw]{A}
\State[1]{(3,0)}{B}
\State[\varphi]{(3,-3)}{F}
\State[-\varphi]{(-3,-3)}{G}
\StateVar[\varphi-1]{(6,0)}{C}
\State[-1]{(-3,0)}{D}
\StateVar[-\varphi+1]{(-6,0)}{E}
\Final[se]{A}
\Final[se]{D}
\Final[sw]{B}
\Final[s]{C}
\Final[s]{E}
\Final[e]{G}
\Final[w]{F}
\EdgeL[.5]{A}{B}{1}
\EdgeL[.5]{B}{F}{0}
\EdgeR[.5]{D}{G}{0}
\EdgeL[.5]{B}{C}{\bar 1}
\EdgeR[.5]{A}{D}{\bar 1}
\EdgeR[.5]{D}{E}{1}
\LoopN[.5]{A}{0}

\LoopW[.5]{G}{1}
\LoopE[.5]{F}{\bar 1}
\ArcL[.5]{C}{B}{0}
\ArcR[.5]{E}{D}{0}
\LArcR[.5]{C}{A}{\bar 1}
\LArcL[.5]{E}{A}{1}
\end{VCPicture}%
     }
\end{center}
\caption{Finite B\"uchi automaton recognizing $Z_{\varphi,1}$ for $\varphi=\frac{1 +
\sqrt{5}}{2}$.}
\label{zer-aut-phi}
\end{figure}
\end{example}

\begin{theorem}\label{recog2} 
Let $\beta >1$ and $A = \{-M, \ldots, M\}$ with $M$ a fixed integer $\ge 1$.  
The set $Z_{\beta,M}$ is recognizable by a finite B\"uchi automaton  if and only if the spectrum $S_{  A }(\beta)$  has no accumulation point. 
\end{theorem}

\begin{proof}  To any string  $z=z_{1} z_{2}\cdots \in Z_{\beta,M}$ we assigne the sequence of polynomials 
$P^{(z)}_n(X) = z_1X^{n-1}+z_2X^{n-2}+ \cdots + z_{n-1}X + z_{n}$.  
Denote $R^{(z)}_n$ the remainder  of the Euclidean division of the polynomial $ P^{(z)}_n(X)$   by the polynomial $(X-\beta)$.  
It means that  there exists a polynomial $Q^{(z)}_n(X)$ such that $P^{(z)}_n(X) = (X-\beta)Q^{(z)}_n(X) +R^{(z)}_n$. Clearly $P^{(z)}_n(\beta) = R^{(z)}_n$.   Denote $R=\{ R^{(z)}_n : z \in Z_{\beta,M} \ \text{and}\  n \in \mathbb{N}\}$ .

 As  $z=z_{1} z_{2}\cdots $ is a $(\beta,A)$-representation of $0$, the value $P^{(z)}_n(\beta) =- 0\decdot z_{n+1}z_{n+2}\cdots$ belongs to the spectrum  $S_{ A }(\beta)$  and $-P^{(z)}_n(\beta)$  is the $n^{th}$ tail $r_n$ of the $(\beta,A)$-representation of $0$.  Consequently, 
\begin{equation}\label{recog}R\subset S_{A }(\beta) \quad \text{ and } \quad  R\ \ \text{ is bounded}.
\end{equation} 

To prove the theorem, we apply  Proposition 3.1 from~\cite{Fr1992}. It says that  $Z_{\beta,M}$ is
recognizable by a finite B\"uchi automaton if and only if  the set $R$ is finite.

$(\Leftarrow)$  \ \  If $Z_{\beta,M}$ is not 
recognizable by finite automaton, then $R$ is infinite and by \eqref{recog}   the spectrum has an accumulation point. 

$(\Rightarrow)$  \ \ If $S_{  A }(\beta)$  has an accumulation point, then by Theorem \ref{rigid},  zero  has a rigid representation  $z_{1} z_{2}\cdots  \in Z_{\beta,M}$ . By Point 2 of  Lemma \ref{aux}, the sequence  of  its tails  $(r_n)$ is injective. Since $-r_n=P^{(z)}_n(\beta)= R^{(z)}_n\in R$, the set $R$ is not finite and therefore  $Z_{\beta,M}$ is not 
recognizable by finite automaton. 
\end{proof}

Combining Theorems  \ref{Aki} and \ref{recog2},  we 
 answer a conjecture raised in \cite{FrSak2010} and obtain the following result.
\begin{theorem}\label{final}
Let $\beta>1$. The following conditions are equivalent:
\begin{enumerate}
  \item the set $Z_{\beta,  M}$ is recognizable by a finite B\"uchi automaton for every
posi\-tive integer $M$,
  \item the set $Z_{\beta,  M}$ is recognizable by a finite B\"uchi automaton for one
$M \ge \lceil \beta \rceil -1$,
  \item $\beta$ is a Pisot number.
\end{enumerate}
\end{theorem}

\begin{remark}
The fact that, if $\beta$ is not a Pisot number, then the set
$Z_{\beta,M}$ is not
recognizable by a finite B\"uchi automaton for any $M \ge \lceil \beta \rceil$ was already settled in Theorem~\ref{origin}, but the proof given above
is more direct than the original one.
\end{remark}



\subsection{Normalization}

Normalization in base $\beta$ is the function which maps a $\beta$-representation on the canonical alphabet
$A_\beta=\{0,\ldots,\lceil \beta\rceil-1\}$ of a number $x \in [0,1]$ onto the greedy
$\beta$-expansion of $x$. From the B\"uchi automaton $\mathcal{Z}$ recognizing the set of
representations of $0$ on the alphabet $\{-\lceil \beta\rceil+1, \ldots, \lceil \beta\rceil-1\}$, one constructs a 
B\"uchi automaton (a converter) $\mathcal{C}$ on the alphabet $A_\beta \times A_\beta$ that recognizes the set of couples on $A_\beta$
that have the same value in base $\beta$, as follows:
$$s \stackrel{(a,b)}{\longrightarrow} t \; \;  \textrm{in } \mathcal{C} \iff
s \stackrel{a-b}{\longrightarrow} t \; \;  \textrm{in } \mathcal{Z},$$
see~\cite{FrSak2010} for details.
Obviously $\mathcal{C}$ is finite if and only if $\mathcal{Z}$ is finite.

Then we take the intersection of the set of second components with
the set of greedy $\beta$-expansions of the elements of $[0,1]$, which
is recognizable by a finite B\"uchi automaton when $\beta$ is a Pisot number, see~\cite{BertrandMathis86}.
Thus
the following result holds true.

\begin{corollary}
Normalization in base $\beta>1$ is computable by a finite  B\"uchi automaton
on the alphabet $A_\beta \times A_\beta$ if and only if $\beta$ is a Pisot number.
\end{corollary}


\section{On-line division in complex base}

\subsection{Trivedi-Ercegovac algorithm}

On-line arithmetic, introduced in~\cite{TrivediErcegovac}, is a mode of computation where operands and results are processed
 in a digit serial manner, starting with the most significant digit. To generate the first digit of the result, the first $\delta$ digits of the operands are required. The integer $\delta$ is called the delay of the algorithm.

In \cite{BrFrPeSv1, BrFrPeSv2} we have extended the original on-line algorithm of Trivedi-Ercegovac to the complex case.

The algorithm for on-line division in a complex numeration system $(\beta,A)$ has two parameters: the delay $\delta\in \N$ and  $D  >0$,  the minimal value (in modulus) of the divisor. 

The $(\beta,A)$-representation of the nominator is $n = \sum_{i=1}^\infty n_i \beta^{-i}$, of the divisor is $d = \sum_{i=1}^\infty d_i \beta^{-i}$, and of their quotient $q = \sum_{i=1}^\infty q_i \beta^{-i}$. Partial sums are denoted by $N_k =\sum_{i=1}^k n_i \beta^{-i}$, $D_k = \sum_{i=1}^k d_i \beta^{-i}$, and $Q_k = \sum_{i=1}^k q_i \beta^{-i}$.

The inputs of the algorithm are two infinite strings
$
0\decdot n_1 n_2 \cdots n_\delta n_{\delta +1} n_{\delta +2} \cdots $ with  $n_i \in A$  and $n_1 = n_2= \cdots = n_\delta = 0$  and
$    0\decdot d_1 d_2  \cdots $ with $d_i \in A$ satisfying $|D_j| \geq D $ for all $j \in \N$, $j\geq 1$.

The output is a string $ q_1 q_2 q_3 \cdots$ corresponding to a $(\beta,A)$-represen\-tation of the quotient $q=n/d= 0\decdot q_1 q_2 q_3 \cdots$. The settings of the algorithm ensure that the representation of $q$ starts behind the fractional point.

Set $W_0 = q_0 = Q_0 = 0$. Then, for $k \ge 1$ compute
$$    W_{k} = \beta (W_{k-1} - q_{k-1} D_{k-1+\delta}) + (n_{k+\delta} - Q_{k-1} d_{k+\delta}) \beta^{-\delta}.$$

The $k$-th digit $q_k$ of the representation of the quotient is evaluated by a function $\Select$, function of the values of the auxiliary variable
$W_k$ and the interim representation $D_{k+\delta}$, so that
$$
    q_k=\Select (W_k, D_{k+\delta}) \in A \, .
$$

It can be shown that for any $k \geq 1$:
$$
    W_k = \beta^k (N_{k+\delta} - Q_{k-1} D_{k+\delta}) \, .
$$
Moreover, if the sequence $(W_k)$ is bounded, then $ q=\lim_{k\to \infty}{Q_k} = \frac{n}{d} \, . $

Conditions on the system $(\beta,A)$ so that the definition of the function $\Select$ ensures the correctness of the on-line division algorithm are given in~\cite{BrFrPeSv1, BrFrPeSv2}.

\subsection{Preprocessing of divisors}\label{prediv}

When making division, we need that the divisor stays away from $0$.
By definition of the on-line algorithm, this means that the value of all the prefixes of the divisor
$d_1d_2 \cdots$ must be greater in absolute value than a parameter $D>0$.

\begin{definition}\label{DefPrepro}
We say that  a complex numeration system $(\beta,A )$  \emph{allows preprocessing}
if
there exists  $D>0$ and a finite list $\mathcal{L}$ of identities of the type  $0\decdot w_{k}\cdots w_0 = 
0\decdot 0u_{k-1}\cdots u_0$ with digits in $A$ such that any string 
$d_1d_2\cdots$ on $A$ without prefix  $w_{k} \cdots w_0$ from $\mathcal{L}$ satisfies 
$| 0\decdot  d_1d_2 \cdots d_j | >D$ for all $j \in \mathbb{N}$. 
\end{definition}

We must have at least $d_1 \neq 0$ after preprocessing, so the preprocessing consists first of all in shifting the fractional
point to the most significant non-zero digit of the $(\beta,A )$-representation of the divisor.
Of course, after preprocessing the value of the original divisor $w$ has been changed
into a new one $d$ which is just a shift of the original one, that is to say $d=w \beta^k$ for some $k$ in $\Z$.
This will have to be taken into account to give the result of the division.

If  zero  has only the
trivial $(\beta,   A )$-representation the situation is simple. This fact  can be equivalently rewritten  as 
$$ \inf\mathcal{R} >0, \ \ \text{where } \ \  \mathcal{R}=\Bigl\{\bigl|  \sum_{i\geq 1}z_i\beta^{-i} \bigr|\  :\   z_1\neq 0, z_i \in   A \Bigr\}\,. $$
In this case the numeration system $(\beta,   A )$  allows preprocessing, since we can take  $D = \inf \mathcal{R}$ 
and the list of rewriting rules is empty. 

\begin{example}
If $\beta =4$  and $   A  = \{\overline{2},\overline{1},  0,1,2\}$, then zero has only the trivial representation and for $D$  one can take $\tfrac{1}{12} = \min \mathcal{R}$.
\end{example}

\begin{example}\label{binary}
 If $\beta=2$ and  $  A  = \{\overline{1},0,1\}$,  zero has two  non-trivial representations $0=0\decdot1\overline{1}\,\overline{1}\,\overline{1}\,\overline{1}\cdots = 0\decdot  \overline{1}\,1\,1\,1\,1\cdots $. Therefore,  preprocessing is a little bit  more sophisticated. Consider the list 
$$   0 \decdot\overline{1}1 =0\decdot 0\overline{1} \ \text{ and} \ \   0 \decdot 1\overline{1} =0\decdot 0 1$$ 
If a  string $d_1d_2\cdots$ has no  prefix  $\overline{1}1$ neither $1\overline{1}$, then
$$|0\decdot d_1d_2\cdots d_j|\geq 0\decdot  10\overline{1}\,\overline{1}\,\overline{1}\cdots = \tfrac14 $$
and thus one can take $D=\tfrac14$. 
\end{example}

\begin{theorem}\label{prepro} 
A complex numeration system $(\beta,A )$ allows preprocessing if and only if the spectrum $S_{{A}}(\beta)$ has no  accumulation point. 
\end{theorem}
The result is proved by the following three lemmas, in which we use the notation

$$H=\max\{|\sum_{i\geq 1}d_i\beta^{-i}|\  : \ d_i \in A \ \text{ for all } i \in \mathbb{N}  \}.$$

\begin{lemma}\label{1} If $0$ has a rigid $(\beta, \mathcal{A})$-representation then the  numeration system $(\beta, \mathcal{A})$   does not allow preprocessing.
\end{lemma}

\begin{proof}  Let  $0 = 0\decdot z_1z_2 z_3\cdots $  be a rigid representation of $0$.  Assume that preprocessing is possible with $D>0$. 
Find $j$ such that $ \tfrac{H}{|\beta|^{j}} < D$. Consider the number $0\decdot z_1z_2 z_3\cdots z_j000\cdots$. Since the representation of zero is rigid, no prefix of the string $z_1z_2 z_3\cdots z_j$ is contained in the list of the rewriting rules. But 
 $| 0\decdot z_1z_2 z_3\cdots z_j| = |0\decdot \underbrace{00\cdots 0}_{j-\text{times}} z_{j+1}z_{j+2}\cdots | < \tfrac{H}{|\beta|^{j}} <D$ --- a contradiction. 
\end{proof}

\begin{lemma}\label{0}
 Let us assume that  $S_{  A }(\beta)$ has no accumulation point and fix $K >0$.  Then there exists $m\in \mathbb{N}$  such that  any string $x_{m-1}x_{m-2}\cdots x_1x_0$ of length $m$ over $  A $ satisfies 
 either 
 $$  |x_{m-1}\beta^{m-1}+x_{m-2}\beta^{m-2}+\cdots +x_1\beta +x_0| \geq  K$$
 or there exists a string $y_{k-1}x_{k-2}\cdots y_1y_0$ of length $k<m$  over $  A $ such that
$$x_{m-1}\beta^{m-1}+x_{m-2}\beta^{m-2}+\cdots +x_1\beta +x_0 = y_{k-1}\beta^{k-1}+y_{k-2}\beta^{k-2}+\cdots +y_1\beta +y_0\,.$$
\end{lemma}

\begin{proof} Since  $S_{  A }(\beta)$ has no accumulation point, the set $P= \{ z \in  S_{  A }(\beta) : |z| < K\}$ is finite. Denote $m=1 + \max\{\rho(z) : z \in P\}$.   
Let $x =x_{m-1}\beta^{m-1}+x_{m-2}\beta^{m-2}+\cdots +x_1\beta +x_0$.  Obviously, $x \in S_{  A }(\beta)$.   Then either $|x| \geq K$ or $x \in P$ and thus $x =  y_{k}\beta^{k-1}+y_{k-2}\beta^{k-2}+\cdots +y_1\beta +y_0$, where $ k \leq \max\{\rho(z) : z \in P\} \leq m-1$. 
\end{proof}

\begin{lemma}\label{3} If  $S_{  A }(\beta)$ has no accumulation point, then there exists $D>0$  and $m \in \mathbb{N}$ such that  for all infinite strings $d_1d_2\cdots$ over $  A $ one has
\begin{enumerate}

\item either   \ \   $| 0\decdot d_1d_2 \cdots d_j| \geq D$ for all    $j \in \mathbb{N}$, 

\item or  \ \ $0\decdot d_1d_2 \cdots d_m= 0\decdot 0  d'_{2}  d'_{3}  \cdots  d'_{m} $ \  for some string  \ $ d'_{2}  d'_{3}  \cdots  d'_{m} \in   A ^*$. 

\end{enumerate}
\end{lemma}
\begin{proof} Let us take $\mu >0$ and  apply Lemma~\ref{0} with $K=H+\mu$ to get $m \in \mathbb{N}$. Denote  
$\mathcal{D} = \{|0\decdot d_1d_2\cdots d_j| : j< m \ \text{and \ } \    0\decdot d_1d_2\cdots d_j \neq 0\decdot 0 d'_2\cdots  d'_j\}$. The set  $\mathcal{D}$ is finite and does not contain   zero. Therefore,   $D'=\min \mathcal{D}>0$. 

To prove  the lemma, consider an infinite string  $d_1d_2\cdots$  and assume that $0\decdot d_1d_2 \cdots d_m\neq  0\decdot 0  d'_{2}  d'_{3}  \cdots  d'_{m} $ for all strings  \ $ d'_{2}  d'_{3}  \cdots  d'_{m} \in   A ^*$. We distinguish two cases
\begin{itemize}
 \item $j<m$, $j\in \mathbb{N}$.  Then $ 0\decdot d_1d_2 \cdots d_j \neq 0\decdot 0 d'_2\cdots  d'_j$, otherwise $0\decdot d_1d_2 \cdots d_m= 0\decdot 0  d'_{2}  d'_{3}  \cdots  d'_{j}d_{j+1}\cdots d_m $ ---  a contradiction. Therefore, $| 0\decdot d_1d_2 \cdots d_j | \geq    D'$.
     
 \item $j\geq m, j\in \mathbb{N}$. Then
$$  |0\decdot d_1d_2\cdots d_j| \geq |0\decdot d_1d_2\cdots d_m| - \tfrac{1}{|\beta|^m} |0\decdot d_{m+1}d_{m+2}\cdots d_{j} |\geq \tfrac{1}{|\beta|^m}  K - \tfrac{1}{|\beta|^m}   H = \frac{\mu}{|\beta|^m}$$

\end{itemize}

Thus we can set $D=\min\Bigl\{D', \frac{\mu}{|\beta|^m} \Bigr\}$.
\end{proof}

The previous lemma gives a hint for creating the list of rewriting rules. We take the index $m$ found by the lemma and 
inspect all strings $d_1d_2\cdots d_m$ over $A$. If 
 $0\decdot d_1d_2 \cdots d_m= 0\decdot 0  d'_{2}  d'_{3}  \cdots  d'_{m} $   for some  string $d'_{2}  d'_{3}  \cdots  d'_{m}$ we put it into the list. 
\begin{example}  Let $\beta = \varphi = \tfrac{1+\sqrt{5}}{2}$  and $  A  = \{\overline{1},0,1\}$. The minimal polynomial of $\varphi$ is  $X^2 - X -1$.  In this numeration system, $0$ has countably many finite representations and uncountably many infinite representations.  
As the alphabet is symmetric, the rewriting rules appear in pairs. For example, as $10\overline{1}$ can be rewritten to  $010$, also  $\overline{1}01$ can be rewritten to $0\overline{1}0$. To shorten our list, we  put into it only one rule of each  pair, namely the rule, where the first digit is $1$.   First we consider the list 

$\mathcal{L}_0:    \quad 10\overline{1} \longrightarrow  010\ ,\  \ 1\overline{1}0 \longrightarrow  001\ ,\  \
 1\overline{1}\overline{1}\longrightarrow  000 $.

\medskip

\noindent {\bf Claim:}  If no rule from $\mathcal{L}_0$ can be applied to the string $d_1d_2\cdots$, then
$|d|\geq  D= \tfrac{1}{\varphi^5}$, where  
$d=0\decdot d_1d_2\cdots $. 

\begin{proof}  WLOG $d_1 = 1$.  

If  $d_2 = 0$, then   $\geq 0$ and thus   $|d| \geq \tfrac{1}{\varphi} - \sum_{k\geq 4}^\infty\varphi^{-k} = \tfrac{1}{\varphi} - \tfrac{1}{\varphi^2} = \tfrac{1}{\varphi^3}\geq D\,.$

If  $d_2 = 1$, then
$|d|\geq \tfrac{1}{\varphi} + \tfrac{1}{\varphi^2}  -  \sum_{k\geq 3}^\infty\varphi^{-k} = 1 - \tfrac{1}{\varphi} -  = \tfrac{1}{\varphi^2}\geq D\,.$

If  $d_2 = \overline{1}$, then $d_3 = 1$.  Therefore,
$|d| \geq \tfrac{1}{\varphi} - \tfrac{1}{\varphi^2}  +\tfrac{1}{\varphi^3}  -  \sum_{k\geq 4}^\infty\varphi^{-k}  = \tfrac{1}{\varphi^5}\geq D\,.$
\end{proof}

\noindent  We can extend the list of rewriting rules to increase the lower bound $D$.   Let us consider the whole families of rules 
$\mathcal{L}$ :

$ (1\overline{1})^k 0 \longrightarrow  00(10)^{k-1}1$  \quad for $k\geq 1$.

$ (1\overline{1})^k \overline{1}\longrightarrow  00(10)^{k-1}0$ \quad   for $k\geq 1$.

$ (1\overline{1})^k 10\overline{1}\longrightarrow  01(00)^{k}0$ \quad  for $k\geq 0$.

$ (1\overline{1})^k 100\longrightarrow  01(00)^{k}1$ \quad  for $k\geq 0$.

$ (1\overline{1})^k 11\longrightarrow  01(00)^{k-1}10$ \quad  for $k\geq 1$.

$ 10\overline{1}^k 0\longrightarrow  0^{k+1}11$  \quad  for $k\geq 0$.

$ 10\overline{1}^k 1\longrightarrow  0^{k}101$ \quad  for $k\geq 1$.

\medskip

\noindent {\bf Claim:}  If no rule from $\mathcal{L}$ can be applied to the string $d_1d_2\cdots$, then
$|d|\geq  D=\tfrac{1}{\varphi^2}$, where  
$d=0\decdot d_1d_2\cdots $. 

\begin{proof}  WLOG $d_1 = 1$.  
Our string has a prefix  $11$ or a prefix $ (1\overline{1})^k101$  for $k\geq 0$. Therefore either 
$$|d| \geq 0\decdot  11 (\overline{1})^\omega = \tfrac1{\varphi^2}\qquad \text{or}
\qquad  |d|\geq 0\decdot  (1\overline{1})^k101 (\overline{1})^\omega = \tfrac1{\varphi^2}+ \tfrac1{\varphi^{2k+3}}.$$
\end{proof}
\end{example}

Some examples where the base is a complex number can be found in~\cite{BrFrPeSv2}.


\section{Comments and open questions}
\subsection{F-number}
In \cite{Lau1993} Lau defined for $1 < \beta < 2$ the following notion, that we extend to any $\beta>1$.

\begin{definition}
Let $\beta>1 $ and $B_\beta=\{-\lceil \beta \rceil +1, \ldots,\lceil \beta \rceil -1\}$ be the symmetrized alphabet of the canonical
alphabet $A_\beta$.
Then $\beta$ is said to be a \emph{F-number} if the set
$$L_{(\lceil \beta \rceil -1)} (\beta)=  S_ {B_\beta}(\beta) \cap \Big[-\frac{\lceil \beta \rceil -1}{\beta-1},\frac{\lceil \beta \rceil -1}{\beta-1}\Big]$$
is finite.
\end{definition}

Feng proved in \cite{Feng} that $1 < \beta < 2$ is a F-number if and only if it is a Pisot number.
This property extends readily to any $\beta>1 $. Another way of proving it consists in realizing that
the set of states $Q_{(\lceil \beta \rceil -1)}$ of the automaton for $Z_{\beta, \lceil \beta \rceil -1}$ is  included
into $L_{(\lceil \beta \rceil -1)}(\beta)$.

\subsection{Open questions} 

\begin{itemize}

\item A motivation for introducing the notion of ``rigid representation of zero" comes from on-line division in a numeration system $(\beta, A)$. A more elementary question is ``Has zero a non-trivial  $(\beta, A)$-representation"? The answer is easy  for real bases and  alphabets of the form $A=\{m, \ldots, 0,  \ldots, M \}$, see Proposition~\ref{non-sym}. 
The same question for complex bases is an open problem.

\item In the case that the base is real and the alphabet is $A=\{-M, \ldots,  M\}$, Theorem \ref{recog2} says that recognizability by a finite automaton is equivalent to the fact that the spectrum  $S_A(\beta)$ has no accumulation point.
  
An analogous result can be proved for complex bases as well. But for complex bases the question about the existence of accumulation points in the spectrum   $S_A(\beta)$ is not yet investigated. 
Nevertheless, it is often easy to check that a $(\beta, A)$-representation of $0$ is \textbf{not}
rigid.

\item  If $\beta >1$ is a non-Pisot base then $A=\{-\lceil \beta \rceil +1, \ldots, \lceil \beta \rceil -1\}$ is the smallest symmetric alphabet of consecutive integers for which the spectrum $S_A(\beta)$ has an accumulation point. What is the minimal size of an  alphabet $A=\{-M, \ldots,  M\}\subset \mathbb{Z}$  for which the spectrum of a non-Pisot complex number $\beta$ has an accumulation point? 
\end{itemize}


\section*{Acknowledgements}
Ch.~F. acknowledges support of ANR/FWF project ``FAN", grant ANR-12-IS01-0002.

\end{document}